\documentclass[11pt, english]{smfart}
\usepackage[T1]{fontenc}
\usepackage[english,francais]{babel}
\usepackage{graphicx}
\usepackage{amsmath,amssymb}
\usepackage[new]{old-arrows}
\usepackage{mathtools}
\usepackage{tikz}
\usepackage{tikz-cd} 
\usetikzlibrary{matrix}
\usepackage{smfthm}
\usepackage{mathrsfs}
\usepackage{pb-diagram}
\usepackage{yfonts}
\usepackage[inline]{enumitem}
\usepackage{color}
\usepackage{hyperref}
\usepackage{verbatim}
\usepackage{dsfont}
\usetikzlibrary{arrows, matrix} 

\def\R{\mathbb{R}}
\def\Q{\mathbb{Q}}

\def\Z{\mathbb{Z}}
\def\N{\mathbb{N}}
\def\O{\mathcal{O}}
\def\L{\Lambda}

\def\nm{\lVert\cdot\rVert}

\def\deg{{\mathrm{deg}}}
\def\vol{\widehat{\mathrm{vol}}}
\def\adlcv{(K,(\Omega, \mathcal A,\nu),\phi)}

\def\Div{\mathrm{Div}}
\def\mumin{\widehat{\mu}_{\min}}
\def\mumax{\widehat\mu_{\max}}

\def\ot{\otimes}

\def\ardeg{\widehat{\mathrm{deg}}}

\def\ovl{\overline}

\def\quot{\mathrm{quot}}
\def\xanomega{X^{\mathrm{an}}_\omega}

\usepackage[normalem]{ulem}
\definecolor{mred}{rgb}{0.83, 0.0, 0.0}
\definecolor{darkspringgreen}{rgb}{0.09, 0.45, 0.27}
\definecolor{ruby}{rgb}{0.88, 0.07, 0.37}
\def\colorsout#1{\bgroup\markoverwith{\textcolor{#1}{\rule[0.5ex]{2pt}{0.7pt}}}\ULon} 
\def\coloruline#1{\bgroup\markoverwith{\textcolor{#1}{\rule[-0.5ex]{2pt}{0.7pt}}}\ULon} 

\makeatletter
\newcommand\tint{\mathop{\mathpalette\tb@int{t}}\!\int}
\newcommand\bint{\mathop{\mathpalette\tb@int{b}}\!\int}
\newcommand\tb@int[2]{%
  \sbox\z@{$\m@th#1\int$}%
  \if#2t%
    \rlap{\hbox to\wd\z@{%
      \hfil
      \vrule width .35em height \dimexpr\ht\z@+1.4pt\relax depth -\dimexpr\ht\z@+1pt\relax
      \kern.05em 
    }}
  \else
    \rlap{\hbox to\wd\z@{%
      \vrule width .35em height -\dimexpr\dp\z@+1pt\relax depth \dimexpr\dp\z@+1.4pt\relax
      \hfil
    }}
  \fi
}
\makeatother

\usepackage{etoolbox}
\makeatletter
\newcommand*\suppresschapternumber{%
  \let\@makechapterhead\@makeschapterhead
  \patchcmd{\@chapter}
    {\protect\numberline{\thechapter}}
    {}
    {}{}%
}
\newcommand*\removedotbetweenchapterandsection{%
  \renewcommand\thesection{\thechapter\@arabic\c@section}%
}
\makeatother

\newtheorem{theorem}{Theorem}

\title{Continuous extension and Birational invariance of $\chi$-volume over an adelic curve}
\author{Wenbin LUO\thanks{
This work was supported by JSPS KAKENHI Grant Number JP20J20125.}}
\email{w.luo@math.kyoto-u.ac.jp}
\address{Department of Mathematics, Graduate School of Science, Kyoto university, Kyoto 606-8502, Japan}
\date{\today}

\begin{document}

\maketitle
\begin{abstract}
By giving an estimate on the minimal slopes, we prove a Hilbert-Samuel formula for semiample and semipositive adelic line bundles. We also show the birational invariance of the arithmetic $\chi$-volume and its continuous extension on the semiample cone.
\end{abstract}
\tableofcontents
\section*{Introduction}
\subsection*{Backgrounds}
Arakelov geometry is a theory to study varieties over $\mathcal O_K$ where $K$ is a number field. Since the closed points of $\mathrm{Spec}\mathcal O_K$ only give rise to non-Archimedean places of $K$, the main idea of Arakelov Geometry is to "compactify" $\mathrm{Spec}\mathcal O_K$ by adding Archimedean places, and the corresponding "fibers" are nothing but analytification of the generic fiber. In order to generalize the theory to the case over a function field or even more general cases, Chen and Moriwaki established the theory of Arakelov geometry over adelic curves\cite{adelic}. Here an adelic curve is a field equipped with a set of absolute values parametrised by a measure space. Such a structure can be easily constructed for any global fields by using the theory of adèles. Moreover, we can define an adelic curve for a finitely generated field over $\Q$ over which Moriwaki introduced the height function in \cite{moriwaki2000arithmetic}. 

The theory over adelic curves thus provide a unified way to consider arithmetics over these fields. The first step is to give a theory of geometry of numbers. Inspired by the study of semistability of vector bundles over projective regular curves, the arithmetic slope theory was introduced in \cite{adelic}, which was established over number fields in \cite{BOST2013}. 

In this paper, we mainly discuss the so-called $\chi$-volume which is the leading coefficient of the arithmetic Hilbert-Samuel function over adelic curves, including its birational invariance and continuous extension on the semiample cone. In the case over number fields, the $\chi$-volume was studied in \cite{moriwaki2009convol,ikoma_boundedness,Yuan_2008}. We has gave some properties of $\chi$-volume over function fields in previous paper \cite{luo2021relative,luo2022slope}.

\subsection*{Arakelov geometry over adelic curves}
Let $\adlcv$ be a proper adelic curve, where $\phi$ is a map from $\Omega$ to the set $M_K$ of all absolute values on $K$. The properness of an adelic curve is equivalent to say it satisfies a product formula (see \ref{subsect_def_adlcv}). For each $\omega\in \Omega$, we denote by $K_\omega$ the completion of $K$ with respect to the corresponding absolute value. Let $E$ be a vector space over $K$ of dimension $n$. Let $\xi=\{\nm_{\omega}\}$ be a norm family where each $\nm_\omega$ is a norm on $E_{K_\omega}:=E\otimes_K K_\omega$. 
We say the pair $\overline E=(E,\xi)$ is an \textit{adelic vector bundle} if $\xi$ satisfies certain dominancy and measurability properties\cite[4.1.4]{adelic}. If $E\not=0$, then we define its \textit{Arakelov degree} as
$$\ardeg(\overline E):=-\int_\omega\ln \lVert s\rVert_{\omega,\mathrm{det}} \nu(d\omega)$$
where $s\in \mathrm{det}E\setminus \{0\}$ and each $\nm_{\omega,\mathrm{det}}$ is the determinant of $\nm_{\omega}$.

Let $X$ be a $K$-projective scheme of dimension $d$. For each $\omega\in\Omega$, we denote by $X^{\mathrm{an}}_\omega$ the analytification of $X_\omega:=X\times_{\mathrm{Spec}K}\mathrm{Spec}K_\omega$ with respect to $\lvert\cdot\rvert_\omega$ in the sense of Berkovich. Consider a pair $\overline L=(L,\{\varphi_\omega\}_{\omega\in \Omega})$ such that $L$ is a line bundle on $X$ and each $\varphi_\omega$ is a metric of the analytification of $L$ on $X^{\mathrm{an}}_\omega$. If the metric family is dominated and measurable\cite[6.1]{adelic}, we say $\ovl L$ is an adelic line bundle.

We assume that $X$ is geometrically reduced. We denote by $\nm_{\varphi_\omega}$ the supnorm on $H^0(X,L)\ot K_\omega$, and by $\xi_{\varphi}$ the norm family $\{\nm_{\varphi_\omega}\}$. 
Now we assume that either the $\sigma$-algebra $\mathcal A$ is discrete, or the field $K$ admits a countable subfield which is dense in every $K_\omega$ with respect to $\lvert\cdot\rvert_\omega$ for every $\omega\in\Omega$. Then $\pi_*(\ovl L):=(H^0(X,L),\xi_\varphi)$ is an adelic vector bundle \cite[Theorem 6.2.18]{adelic}. Hence we can define the $\chi$-volume function by
$$\vol_\chi(\overline L):=\limsup_{n\rightarrow +\infty}\frac{\ardeg(\pi_*(n\ovl L))}{n^{d+1}/(d+1)!}.$$
The superior limit in the definition can be replaced by limit when the asymoptotic minimal slope $\displaystyle\mumin^{\inf}(\ovl L)=\liminf\limits_{n\rightarrow +\infty}\frac{\mumin(\pi_*(n\ovl L))}{n}$ is finite. We refer the reader to \ref{subsec_adl_vb} for the definition of minimal slopes. 

\subsection*{Minimal slopes and Arithmetics over trivially valued fields} An $\R$-filtration $\mathcal F^t$ of a vector spaces $E$ gives an ultrametic norm $\lVert\cdot\rVert_{\mathcal F}$ over the trivial absolute value. In particular, the minimal slope of an utrametically normed space $(E,\lVert\cdot\rVert_{\mathcal F})$ can be given by the supremum of norms. This allow us to give the estimate of minimal slope associated to the an ultrametrically normed linear series over a trivially valued field. In conclusion, we prove that for adelic line bundles $\ovl L_1=(L_1,\phi_1),\cdots,\ovl L_r=(L_r,\phi_r)$ such that $L_1,\dots,L_r$ are semiample, there exists constants $S$, $T$ such that
$$\mumin(\pi_*(a_1 \ovl L_1+\cdots+a_r\ovl L_r))\geq S(a_1+\cdots +a_r)+T$$
for $a_1,\dots,a_r\in\N$.

Based on this result, we prove a arithmetic Hilbert-Samuel formula for semiample and semipositive adelic line bundle:

\begin{theorem}[cf. Theorem \ref{theo_semi_hil_sam}]
    Let $\ovl L$ be a semiample and semipositive adelic line bundle. 
    It holds that $$\vol_\chi(\ovl L)=\widehat c_1(\ovl L)^{d+1}$$
\end{theorem}
Moreover, we show that $\vol_\chi(\cdot \ovl L)$ is invariant under a birational morphism $f:X'\rightarrow X$ if $\mumin^{\inf}(f^*\ovl L)\in\R$ without using Stein-factorization(see Theorem \ref{theo_bir_inv}).

\subsection*{Continuous extension of $\vol_\chi(\cdot)$}
Assume that $X$ is further geometrically integral and normal. By the correspondence between continuous metrics and Green functions, we can extend the definition of $\vol_\chi(\cdot)$ to $\R$-Cartier adelic divisors (see section \ref{sect_adl_R_div}). We show that $\displaystyle\frac{\vol_\chi(\ovl D)}{\mathrm{vol}(D)}$ is a concave function on the ample cone of a projective curve or a projective toric varieties. The following are some results we obtained about the continuity of $\vol_\chi(\cdot)$.
\begin{theorem}[cf. Theorem \ref{theo_chi_semiample_con}]
Let $\overline D_1=(D,g),\dots, \overline D_r=(D_r,g_r)$ be adelic $\Q$-Cartier divisors on $X$ such that $D_i$ are semiample. Then $\vol_\chi(\cdot)$ can be continuously extended to a function on the polyhedral cone $$\{a_1 \ovl D_1+\cdots +a_r \ovl D_r\mid (a_1,\cdots, a_r)\in \R_{\geq 0}^r\}.$$
\end{theorem}
\subsection*{Organization of the paper}
We first give preliminaries on the theory of adelic curves, including Arakelov degrees and Harder-Narasimhan filtrations of adelic vector bundles. In section \ref{sect_arith_triv}, we use the ultrametric norm over a trivially valued field to study the fitered vector spaces and linear series. The results can be used to give descriptions of the $\chi$-volume. We deal with the case of adelic line bundles and adelic Cartier divisors respectively in section \ref{sec_adelic_lin} and \ref{sect_adl_R_div}.

\section{Preliminaries}
\subsection{Adelic curves}\label{subsect_def_adlcv}
\begin{defi}
Let $K$ be a field and $M_K$ be all its places. An \textit{adelic curve} is a 3-tuple $\adlcv$ where $(\Omega, \mathcal A, \nu)$ is a measure space consisting of the space $\Omega$, the $\sigma$-algebra $\mathcal A$ and the measure $\nu$, and $\phi$ is a function $(\omega\in\Omega)\mapsto \lvert\cdot\rvert_\omega\in M_K$ such that
$$\omega\mapsto \ln\lvert \alpha\rvert_\omega$$
is $\nu$-integrable for any non-zero element $\alpha$ of $K$. Moreover, if $\displaystyle\int_\Omega \ln\lvert \alpha\rvert_\omega\nu(d\omega)=0$ for any $\alpha\in K^\times$, then we say the adelic curve is \textit{proper}. For each $\omega\in\Omega$, we denote by $K_\omega$ the completion field of $K$ with respect to $\lvert\cdot\rvert_\omega$. We denote that $\Omega_\infty:=\{\omega\in\Omega\mid \lvert\cdot\rvert_\omega\text{ is Archimedean}\}.$
\end{defi}

Some typical examples are number fields, projective curves, polarised varieties. We refer to \cite[Section 3.2]{adelic} for detailed constructions. From now on, we assume the adelic curve is proper unless it's specified. Moreover, for technical reasons, we assume that $\adlcv$ satisfies any of the following two conditions:
\begin{enumerate}
    \item[\textnormal{(i)}] There exists a countable subfield $K'\subset K$ which is dense in every $K_\omega$.
    \item[\textnormal{(ii)}] The $\sigma$-algebra $\mathcal A$ is discrete.
    \item[\textnormal{(iii)}] For every $\omega\in\Omega_\infty$, $\lvert\cdot\rvert_\omega=\lvert\cdot\rvert$ on $\Q$ where $\lvert\cdot\rvert$ is the common absolute value.
\end{enumerate}
It is worth to be mentioned that the condition (iii) guarantees that $\nu(\Omega_\infty)<+\infty$ due to \cite[Proposition 3.1.2]{adelic}.
\subsection{Adelic vector bundles}\label{subsec_adl_vb}
Let $E$ be a vector space over $K$ of dimension $n$. Let $\xi=\{\nm_{\omega}\}$ be a norm family where each $\nm_\omega$ is a norm on $E_{K_\omega}:=E\otimes_K K_\omega$. We can easily define restriction $\xi_F$ of $\xi$ to a subspace $F$, quotient norm family $\xi_{E\twoheadrightarrow G}$ induced by a surjective homomorphism $E\twoheadrightarrow G$, the dual norm family $\xi^\vee$ on $E^\vee$, exterior power norm family and tensor product norm family. Please check \cite[Section 1.1 and 4.1]{adelic} for details.
\begin{defi}[Adelic vector bundles]
We say a norm family $\xi$ is \textit{upper dominated} if
$$\forall s\in E^*, \tint_\Omega \ln\lVert s\rVert_\omega\nu(d\omega)<+\infty.$$
Moreover, we say $\xi$ is \textit{dominated} if its dual norm $\xi^\vee$ on $E^\vee$ is also upper dominated. We say $\xi$ is \textit{measurable} if for any $s\in E^*$, the function $s\mapsto \lVert s\rVert_\omega$ is $\mathcal A$-measurable. If $\xi$ is both dominated and measurable, we say the pair $\overline E=(E,\xi)$ is an \textit{adelic vector bundle}. Note that the property of being an adelic vector bundle is well-preserved after taking restriction to subspaces, quotients, exterior powers and tensor products\cite[Proposition 4.1.32, Remark 4.1.34]{adelic}. It is worth noting that the $\epsilon,\pi$-tensor product $\ovl E_1\ot_{\epsilon,\pi}\ovl E_2$ of adelic vector bundles $\ovl E_1$ and $\ovl E_1$ is defined in the way that the norm on the tensor product are given differently for Archimedean and non-Archimedean places. For the details, we refer the reader to \cite[1.1.11 and 4.1.1.5]{adelic}
\end{defi}

\begin{defi}[Arakelov degrees and slopes]
Let $\overline E=(E,\xi)$ be an adelic vector bundle. If $E\not=0$, then we define its \textit{Arakelov degree} as
$$\deg(\overline E):=-\int_\omega\ln \lVert s\rVert_{\mathrm{det}\,\xi,\omega} \nu(d\omega)$$
where $s\in \mathrm{det}E\setminus \{0\}$ and $\mathrm{det}\,\xi$ is the determinant norm family on $\mathrm{det}E$ (the highest exterior power). Note that this definition is independent with the choice of $s$ since the adelic curve $S$ is proper. If $E=0$, then by convention we define that $\deg(\overline E)=0$.
Moreover, we define the \textit{positive degree} as
$$\deg_+(\overline E):=\sup_{F\subset E}\deg(F,\xi_F),$$
i.e. the supremum of all its adelic vector sub-bundles' Arakelov degerees.
If $E\not=0$, its \textit{slope} $\widehat{\mu}(\overline E)$ is defined to be the quotient $\deg(\overline E)/\dim_K(E)$.
The \textit{maximal slope} and \textit{minimal slope} are defined as
\begin{equation*}\label{def_min_slop}
    \begin{aligned}
    \mumax(\overline E)&:=\begin{cases}\sup\limits_{0\not=F\subset E}\widehat \mu(F,\xi_F), &\text{ if }E\not=0\\
    -\infty, &\text{ if }E=0 \end{cases}\\
    \mumin(\overline E)&:=\begin{cases}\inf\limits_{E\twoheadrightarrow G\not=0}\widehat \mu(G,\xi_{E\twoheadrightarrow G}), &\text{ if }E\not=0\\
    +\infty, &\text{ if }E=0 \end{cases}
    \end{aligned}
\end{equation*}

\begin{prop}\label{prop_deg_posdeg}
Let $\ovl E=(E,\{\nm_{\omega}\})$ be an adelic vector bundle. We give the following properties:
\begin{enumerate}
    \item[\textnormal{(a)}] Let $f:\Omega\rightarrow\R$ be an $\nu$-integrable function. Then $\ovl {E(f)}:=(E,\{\nm_\omega \exp{(-f(\omega)})\})$ is also an adelic vector bundle and \[\begin{cases}
    \displaystyle\ardeg(\ovl {E(f)})=\ardeg(\ovl E)+(\dim_K E)\int_{\Omega} f\nu(d\omega),\\
    \displaystyle\mumin(\ovl {E(f)})=\mumin(\ovl E)+\int_{\Omega} f\nu(d\omega).
    \end{cases}\]
    \item[\textnormal{(b)}] If $\mumin(\ovl E)\geq 0$, then $\ardeg(\ovl E)=\ardeg_+(\ovl E)$.
\end{enumerate}
\end{prop}
\begin{proof}
(a) This is due to the definition.
(b) Assume that there exists a subspace $F\subset E$ such that $\ardeg(\ovl F)>\ardeg(\ovl E)$, then $\widehat\mu(\ovl{E/F})<0$ due to \cite[Proposition 4.3.12]{adelic}, hence a contradiction.
\end{proof}
\end{defi}

\begin{defi}
    We say an adelic vector bundle $\ovl E=(E,\xi=\{\nm_\omega\})$ is \textit{ultrametric} if for every $\omega\in\Omega$ such that $\lvert\cdot\rvert_\omega$ is non-Archimedean, $\nm_\omega$ is ultrametric, that is, 
    $$\lVert s+t\rVert_\omega\leq \max\{\lVert s\rVert_\omega,\lVert t\rVert_\omega\}.$$
\end{defi}
Through the end of this paper, for simplicity, when we say an adelic vector bundle, we mean an ultrametric adelic vector bundle.

\begin{prop}\label{prop_exact}
Let $(E,\xi)$ be an adelic vector bundle where $\dim_K E=r>0$. Consider a flag of vector subspace of $E$:
$$0=E_0\subset E_1\subset\cdots\subset E_n=E.$$ We denote by $\xi_i$ the restriction norm family of $x$ on $E_i$, and $\eta_i$ the quotient norm family of $\xi_i$ on $E_i/E_{i-1}$. The following inequality holds:
$$\begin{aligned}\sum\limits_{i=1}^n \ardeg(E_i/E_{i-1},\eta_i)\leqslant &\ardeg(E,\xi)\\
\leqslant &\sum\limits_{i=1}^n \ardeg(E_i/E_{i-1},\eta_i)+\frac{\nu(\Omega_\infty)}{2}r\ln r.\end{aligned}$$
\end{prop}
\begin{proof}
See \cite[Proposition 4.3.12]{adelic}.
\end{proof}

\begin{rema}
If the field $K$ is perfect, then due to \cite[5.6.2]{adelic}, there exists a constant $C>0$  such that for any adelic vector bundles $\overline E_1,\cdots, \ovl E_l$, it holds that
\begin{align*}
    \mumin(\overline E_1\otimes_{\epsilon,\pi}\cdots \otimes_{\epsilon,\pi}\overline E_l)\geq \mumin(\ovl E_1)+\cdots+\mumin(\ovl E_l)-C\sum_{i=1}^{l} \ln(\dim_K(E_i))
\end{align*}
which was firstly proved in \cite{BOST2013} and reformulated in \cite[Chapter 5]{adelic}. We say called \textit{minimal slope property of level}$\geq C$. All the constant $C$ mentioned in this article refers to the constant in this sense.
\end{rema}
\begin{defi}
Let $\overline E=(E,\xi)$ be an adelic vector bundle of dimension $r$. The \textit{Harder-Narasimhan} $\R$-filtration is given by
$$\mathcal F_{hn}^t(\overline E)=\sum_{\substack{0\not=F\subset E\\\mumin(F,\xi_F)\geq t}}F.$$
We denote by $\mumin(\overline E)=\widehat\mu_{r}(\ovl E)\le \widehat\mu_{n-1}(\ovl E)\le\cdots\le \widehat\mu_1(\ovl E)=\mumax(\overline E)$ the jumping points of the $\R$-filtration.
\end{defi}
\begin{prop}\label{prop_succ_slope}
Let $\ovl E$ be an adelic vector bundle. Then
$$\begin{aligned}\sum \widehat\mu_i(\ovl E)&\le \deg(\overline E)\le \sum \widehat\mu_i(\ovl E)+\frac{1}{2}r\ln r,\\
\sum \max(\widehat\mu_i(\ovl E),0)&\le \deg_+(\overline E)\le \sum \max(\widehat\mu_i(\ovl E),0)+\frac{1}{2}r\ln r.
\end{aligned}$$
\end{prop}

\section{Arithmetic over trivially valued field}\label{sect_arith_triv}
In this section, let $K$ be aribitrary field and $(\lvert\cdot\rvert_0)$ be the trivial absolute value on $K$.
\subsection{Ultrametrically normed vector spaces over $(K,\lvert\cdot\rvert_0)$}
Let $E$ be a finite-dimensional vector space over $K$. Remind that a $\R$-filtration $\mathcal F^t E$ on $E$ is a map
$$t\in\R\mapsto \mathcal F^t E\subset E$$
where each $\mathcal F^t E$ is a subspace of $E$ and $\mathcal F^t E\subset \mathcal F^{t'} E$ if $t\geqslant t'$. Any $\R$-filtration corresponds to a sequence of \textit{successive slopes}, that is,
$$\mu_i=\sup\{t\mid \dim_K\mathcal F^t E\geqslant i\}$$
where $i=1,2,\cdots,\dim_K E$.

From the $\R$-filtration, we can construct an ultrametric norm $\nm_{\mathcal F}$ over $(K,\lvert\cdot\rvert_0)$ which is given by
$$\lVert s\rVert_{\mathcal F}:=\exp(-\lambda_{\mathcal F}(s))$$
where $\lambda_{\mathcal F}(s):=\sup\{t\mid s\in\mathcal F^t E\}.$
This actually gives a bijection between the sets of ultrametic norms and $\R$-filtrations. Moreover, the pair $\ovl E:=(E,\nm_{\mathcal F})$ can be viewed as an adelic vector bundle over $(K,\lvert\cdot\rvert_0).$ In this specific case, its minmial slope can be computed by $\mumin(E,\nm_{\mathcal F})=-\ln(\max\{\lVert x\rVert_{\mathcal F}\mid x\in E\})$. 
In particular, for any subspace $F$ of $E$, we denote by $\nm_F$ the restricted ultrametric norm on $F$, then
$$\mumin(F,\nm_F)\geqslant \mumin(\overline E)$$
\begin{prop}
Let $\overline E=(E,\nm)$ be an ultrametrically normed vector space over $(K,\lvert\cdot\rvert_0)$.
Let $$0\rightarrow F\xrightarrow{f} E\rightarrow G\rightarrow 0$$ be an exact sequence of vector spaces over $K$. Then it holds that
$$\mumin(\overline E)=\min\{\mumin(F,\nm_F),\mumin(G,\nm_{E\twoheadrightarrow G})\}.$$
\end{prop}
\begin{proof}If $E=0$, the proposition is trivial. So we assume that $E\not=0$.
If $\mumin(F,\nm_F)=\mumin(\overline E)$, then we are done. So we assume that $\mumin(F,\nm_F)>\mumin(\overline E)$.
Let $x\in E$ be an element with $\lVert x\rVert=\exp(-\mumin(\overline E))$, that is, an element with the maximal norm.
Since $\mumin(F,\nm_F)>\mumin(\overline E)$, $x$ is not contained in the image of $F$.
Therefore $\tilde x\not=0\in G$, it holds that $$\lVert \tilde x\rVert_{E\twoheadrightarrow G}=\inf\limits_{y\in F}\lVert x+f(y)\rVert=\inf\limits_{y\in F}\max\{\lVert x\rVert, \lVert f(y)\rVert\}=\lVert x\rVert$$
which implies that $\mumin(G,\nm_{E\twoheadrightarrow G})=\mumin(\overline E)$.
\end{proof}
\begin{rema}
The above proposition can be also obtained by using the inequality in \cite[Proposition 4.3.32]{adelic}.
\end{rema}
\begin{defi}
Let ${\overline E_i=(E_i,\nm_i)}_{i=1}^n$ be a collection of ultrametric normed vector spaces over $(K,\lvert\cdot\rvert_0)$. We define the a norm $\nm$ on $\mathop\bigoplus\nolimits_{i=1}^n E_i$ by
$$\lVert (a_1,\cdots,a_n)\rVert:=\max_{i \in \{1, \ldots, n\}}\{\lVert a_i\rVert_i\}$$
for $(a_1,\cdots,a_n)\in\mathop\bigoplus\nolimits_{i=1}^n E_i$, which is called the direct sum of norms $\{\nm_i\}_{i=1}^n$, and the ultrametric normed vector space $(\mathop\bigoplus\nolimits_{i=1}^n E_i,\nm)$ can be denoted by \[\mathop\bigoplus\limits_{i=1,\dots,n}^{\perp}\overline E_i.\]
\end{defi}
\begin{rema}
It's easy to see that $$\mumin\left(\mathop\bigoplus\limits_{i=1,\dots,n}^{\perp}\overline E_i\right)= \min_{i \in \{1, \ldots, n\}}\{\mumin(\overline E_i)\}.$$
Moreover, the successive slopes of $\mathop\bigoplus\limits_{i=1,\dots,n}^{\perp}\overline E_i$ is just the sorted sequence of the union of successive slopes of $\overline E_i$.
\end{rema}
\begin{defi}
Let $\overline E=(E,\nm_E)$ and $\overline F=(F,\nm_F)$ be ultrametric normed vector spaces over $(K,\lvert\cdot\rvert_0)$. The tensor product $\overline E\otimes\overline F$ is defined by equipping $E\otimes F$ with the ultrametric norm
$$\lVert x\rVert_{E\otimes F}:=\min\left\{\max_i\big\{\lVert s_i\rVert_E\cdot\lVert t_i\rVert_F\big\}\big|\text{ }x=\sum_i s_i\otimes t_i\right\} $$
\end{defi}
\subsection{Filtered graded linear series}
\begin{defi}
Let $E_\bullet=\{E_n\}_{n\in\N}$ be a collection of vector subspaces of $K(X)$ over $K$. We say $E_\bullet$ is a \textit{graded linear series} if $\mathop\oplus_{n\in\N}E_n Y^n$ is a graded sub-$K$-algebra of $K(X)[Y]$. If $\mathop\oplus_{n\in\N}E_n Y^n$ is finitely generated $K$-algebra, we say $E_\bullet$ is of \textit{finite type}. We say $E_\bullet$ is of \textit{subfinite} type if $E_\bullet$ is contained in a graded linear series of finite type.

We equip each $E_n$ an ultrametric norm $\nm_n$. The collection $\overline E_\bullet=\{(E_n,\nm_n)\}$ is called an \textit{ultrametrically normed graded linear series}. 
Denote by $\delta:\N\rightarrow \R$ the function maps $n$ to $C\ln \dim_K(E_n)$.
We say $\ovl E_\bullet$ is \textit{$\delta$-superadditive} if we have the following inequality:
$$\prod_{i=1}^r\lVert s_i\rVert_{n_i}\leqslant \left\lVert \prod_{i=1}^r s_i\right\rVert_{n_1+n_2+\cdots+n_r}\prod_{i=1}^r\dim_K(E_{n_i})^C$$
holds for any $r\geqslant 2$, $i=1,\dots,r$, $n_i\geqslant 0$ and $s_i\in E_{n_i}$. 
\end{defi}
\begin{defi}\label{def_asymp}
We define the volume of a graded linear series $E_\bullet$ of Kodaira-dimension $d$ by
$$\mathrm{vol}(E_\bullet):=\limsup_{n\rightarrow +\infty}\frac{\dim_K(E_n)}{n^d/d!}.$$
If $\overline E_\bullet$ is a ultrametrically normed graded linear series satisfying $\delta$-superadditivity, then we define its
\textit{arithmetic volume} and \textit{arithmetic }$\chi$\textit{-volume} by
$$\begin{aligned}\vol(\overline E_\bullet):=\limsup_{n\rightarrow +\infty} \frac{\sum \max(\widehat{\mu}_i(\overline E_n),0)}{n^{d+1}/(d+1)!},\\
\vol_\chi(\overline E_\bullet):=\limsup_{n\rightarrow +\infty}\frac{\sum \widehat{\mu}_i(\overline E_n)}{n^{d+1}/(d+1)!}.
\end{aligned}$$
Its \textit{asymptotic maximal slope}, \textit{lower asymptotic minimal slope}, and \textit{lower asymptotic minimal slope} is defined respectively by
$$\begin{aligned}\mumax^{\mathrm{asy}}(\overline E_\bullet):=\limsup_{n\rightarrow+\infty} \frac{\mumax(\overline E_n)}{n},\\
\mumin^{\inf}(\overline E_\bullet):=\liminf_{n\rightarrow+\infty} \frac{\mumin(\overline E_n)}{n},\\
\mumin^{\inf}(\overline E_\bullet):=\limsup_{n\rightarrow+\infty} \frac{\mumin(\overline E_n)}{n}.
\end{aligned}$$
\end{defi}

\subsection{Arithmetic Okounkov bodies}\label{subsect_Okounkov}
Let $\overline E_\bullet=\{\overline E_n=(E_n,\nm_n)\}_{n\in\N}$ be an ultrametrically normed graded linear series of subfinite type satisfying $\delta$-superadditivity. Let $r_n=\dim_K(E_n)$.
\begin{prop}\label{prop_asy_norm}For any $n\in\N$ and $s\in E_n$, we define
$$\lVert s\rVert'_n:=\liminf_{m\rightarrow \infty}\lVert s^m\rVert_{nm}^{\frac{1}{m}}$$
Then the following four properties holds:
\begin{enumerate}
    \item[\textnormal{(1)}] $\lVert s\rVert'_n=\lim\limits_{m\rightarrow \infty}\lVert s^m\rVert_{nm}^{\frac{1}{m}}$.
    \item[\textnormal{(2)}] $\lVert\cdot\rVert'_n$ is an ultrametric norm on $E_n$ over trivially valued $K$.
    \item[\textnormal{(3)}] $\lVert s\rVert'_n \lVert t\rVert'_m\geqslant \lVert st\rVert'_{n+m}$ for any $s\in E_n$ and $t\in E_m$.
    \item[\textnormal{(4)}] $\displaystyle\frac{\lVert s\rVert'_n}{\lVert s\rVert_n}\leqslant r_n^C .$
\end{enumerate}
\end{prop}
\begin{proof}
(1) Let $a_m:=-\ln\|s^m\|_{nm}$. By the $\delta$-superadditivity, for $m_1,\dots,m_l\in \N_+$, we have
$$a_{m_1+\cdots+m_l}\geqslant a_{m_1}+\cdots+a_{m_l}-\sum_{i=1}^l \delta(nm_i).$$ So the sequence $\left\{\displaystyle\frac{a_m}{m}\right\}$ converges in $\R$ due to a generalized Fekete's lemma \cite[Proposition 6.3.15]{adelic}. 

(2) For any $s,t\in E_n$, $\lVert(s+t)^m\rVert_{nm}\leqslant \max\limits_{i=0,\dots,m}\left\{\lVert s^it^{m-i}\rVert_{nm} \right\}$.
Again by the $\delta$-supperadditivity, we can deduce that $$\lVert s^i t^{m-i}\rVert_{nm}\leqslant \lVert s^i\rVert_{ni}\lVert t^{m-i}\rVert_{n(m-i)}\exp(\delta(ni)+\delta(n(m-i))).$$
Let $A=\max\{\lVert s\rVert'_n,\lVert t\rVert'_n\}$. Then for any $\epsilon>0$, there is an integer $N$ such that \[\begin{cases}\displaystyle
    \lVert s^l\rVert_{nl}^{\frac{1}{l}}\exp(\frac{\delta(nl)}{l})\leqslant A+\epsilon\\
    \displaystyle
    \lVert t^l\rVert_{nl}^{\frac{1}{l}}\exp(\frac{\delta(nl)}{l})\leqslant A+\epsilon
\end{cases}\] for every $l>N$.

Let $B=\max\limits_{i=0,\dots,N}\{\max(\lVert s^i\rVert_{ni}, \lVert t^i\rVert_{ni})\exp(\delta(ni))\}$.
When $m>2N$, either $i$ or $m-i$ is greater than $N$, thus
$$\lVert (s+t)^m\rVert_{nm}^{\frac{1}{m}}\leqslant \max\limits_{i=0,\dots,N}\{(A+\epsilon)^{\frac{m-i}{m}}B^{\frac{1}{m}}\}.$$
The right hand side of the inequality has the limit $A+\epsilon$ which implies that there exists $N'\in\N_+$ such that 
$\lVert(s+t)^m\rVert_{nm}^{\frac{1}{m}}\leqslant A+2\epsilon$ for every $m>N'$.
Therefore $\lVert s+t\rVert'_n\leqslant A$.

(3) Due to the $\delta$-superadditivity, we have $$\frac{-\ln\lVert(st)^l\rVert_{l(n+m)}}{l}\geqslant \frac{-\ln\lVert s^l\rVert_{nl}}{l}+\frac{-\ln\lVert t^l\rVert_{ml}}{l}-\frac{\delta(nl)}{l}-\frac{\delta(ml)}{l}$$ for every $l\in\N_+$. Let $l\rightarrow +\infty$, we obtain (3).

(4) This is a direct result from an estimate of the limit.
\end{proof}
Let $v:K(X)\rightarrow \Z^{d}$ be a valuation of rank $d$ with center at a regular rational point $p\in X$, constructed as in \cite[Theorem 1.1]{chen2018newton}. Here $\Z^d$ is equipped with the lexicographic order. 
We denote that \begin{align*}
    &\Delta(E_\bullet):=\text{convex hull of }\{n^{-1}v(s)|s\in E_n\}\subset \R^{d},\\
    &\Delta^t(E_\bullet):=\text{convex hull of }\{n^{-1}v(s)|s\in E_n,\lVert s\rVert'_n\leq \exp{(-nt)} \}\subset \R^{d}
\end{align*}
for $t\in\R$.
\begin{defi}
We define the \textit{concave transform} $G_{\ovl E_\bullet}:\Delta(E_\bullet)\rightarrow \R$ associated to $\ovl E_\bullet$ as 
$$G_{\ovl E_\bullet}(x):=\sup\{t\in\R\mid x\in \Delta^t(E_\bullet)\}.$$
\end{defi}
\begin{prop}\label{prop_Okounkov} \begin{enumerate}
    \item[\textnormal{(1)}] $G_{\ovl E_\bullet}$ is an upper semi-continuous concave function on $\Delta(E_\bullet).$
    \item[\textnormal{(2)}] $\displaystyle\sup_{x\in \Delta(E_\bullet)}G_{\ovl E_\bullet}(x)=\mumax^{\mathrm{asy}}(\ovl E_\bullet)$ and $\displaystyle\inf_{x\in \Delta(E_\bullet)}G_{\ovl E_\bullet}(x)\geq \mumin^{\mathrm{inf}}(\ovl E_\bullet)$
    \item[\textnormal{(3)}] If $\mumin^{\inf}(\ovl E_\bullet)>-\infty$ and $\mathrm{vol}(E_\bullet)>0$, then $\displaystyle\Big\{\frac{\sum \widehat{\mu}_i(\overline E_n)}{n^{d+1}/(d+1)!}\Big\}_{n\in\N_+}$ converges to $\displaystyle\int_{\Delta(E_\bullet)} G_{\ovl E_\bullet}(x)dx$. In particular,
    $\vol_{\chi}(\ovl E_\bullet)=\displaystyle\int_{\Delta(E_\bullet)} G_{\ovl E_\bullet}(x)dx$.
    \item[\textnormal{(4)}] If $\mumin^{\sup}(\ovl E_\bullet)\geq 0$ and $\mathrm{vol}(E_\bullet)>0$, then $\vol_\chi(\ovl E_\bullet)=\vol(\ovl E_\bullet)$.
    \item[\textnormal{(5)}] Let $\ovl E_\bullet^{(i)}=\{(E_n,\nm_{n}^{(i)})\}$, $i=1,2,3$ be $\delta$-supperadditive ultrametrically normed linear series of subfinite-type. Assume that $$E_n^{(1)}\cdot E_m^{(2)}\subset E_{n+m}^{(3)}$$ and 
    $$\lVert s t\rVert_{n}^{(3)}\leq \lVert s\rVert_n^{(1)}\lVert t\rVert_{n}^{(2)}\dim_K(E_n^{(1)})^C\dim_K(E_n^{(2)})^C$$
    for any $s\in E_n^{(1)},t\in E_n^{(2)}$. Then for any $x\in \Delta(E_\bullet^{(1)}),y\in\Delta(E_\bullet^{(2)})$, we have 
    $G_{\ovl E_\bullet^{(1)}}(x)+G_{\ovl E_\bullet^{(2)}}(y)\leq G_{\ovl E_\bullet^{(3)}}(x+y)$.
\end{enumerate}
\end{prop}
\begin{proof}
    We refer the reader to \cite[6.3.4]{adelic} for the proof of (1)(2)(3)(5). (4) is due to (3) and definition. 
\end{proof}
\subsection{Slope boundedness}
\begin{prop}\label{prop_fin_gen_bound}
Let $\overline E_\bullet=\{\overline E_n=(E_n,\nm_n)\}_{n\in\N}$ be an ultrametrically normed graded linear series over $(K,\lvert\cdot\rvert_0)$ satisfying $\delta$-superadditivity. If $E_\bullet=\{E_n\}_{n\in\N}$ is of finite type, then it holds that
$$\mumin^{\inf}(\overline E_\bullet)>-\infty$$
\end{prop}
\begin{proof}
We fix a set of generators, and assume that they are of degree at most $N$.
Then for any $n\in\N$, we set $$\Psi_n:=\{(\lambda_1,\lambda_2,\cdots,\lambda_r)\mid 1\leqslant\lambda_1\leqslant\lambda_2\leqslant\cdots\leqslant\lambda_r\leqslant N, \sum_{i=1}^r \lambda_i=n, r=1,\cdots, n\}.$$
Then the map $$\bigoplus^\perp_{(\lambda_1,\lambda_2,\cdots,\lambda_r)\in\Psi_n}(\otimes_{i=1}^r (E_{\lambda_i},\nm_{\lambda_i}\dim_K(E_n)^C)\longrightarrow \overline E_n$$
is surjective and of operator norm $\leqslant 1$. Therefore we have $$\mumin(\overline E_n)\geqslant \min\limits_{(\lambda_1,\lambda_2,\cdots,\lambda_r)\in\Psi_n}\{\sum_{i=1}^r (\mumin(\overline E_{\lambda_i})-\delta(\lambda_i))\}.$$
Then $$\frac{\mumin(\overline E_n)}{n}\geqslant \min\left\{\displaystyle\frac{\mumin(\overline E_l)-\delta(l)}{l}\right\}_{1\leqslant l\leqslant N}$$ for every $n\in \N$.
\end{proof}
\subsection{Multiple graded linear series}\label{subsect_multi_lin_ser}
Now let $L_1,\cdots, L_r$ be line bundles over $X$. 
For any $a=(a_1,a_2,\dots,a_r)\in \N^r$, set that \[\begin{cases}\displaystyle a\cdot\textnormal{\textbf{L}}:=\sum_{i=1}^r a_i L_i,\\
\displaystyle a\cdot \varphi:=\sum_{i=1}^r a_i \varphi_i,\\
|a|:=a_1+a_2+\cdots+a_r.
\end{cases}
\]

We may choose suitable local bases of $L_1,\cdots, L_r$ at the regular rational point $p\in X$, so that we can identify $H^0(X,a\cdot \textbf{L})$ with subspaces in $K(X)$. 
We equip each $H^0(X,a\cdot \textbf{L})$ an ultrametric norm $\nm_{(a)}$ over trivially valued field such that for any $a^{(1)},\cdots a^{(r)}\in\N^r$, we have the inequality 
$$\lVert u_1\cdots u_l\rVert_{(a^{(1)}+\cdots+a^{(l)})}\leqslant \prod_{i=1}^l\lVert u_j\rVert_{(a^{(j)})}\dim_K(H^0(X,a^{(j)}\cdot\textbf{L}))^C$$ where $u_j\in H^0(X, a^{(j)}\cdot\textbf{L})$. 
$$\lVert u\rVert_{(a)}\leqslant \prod_{i=1}^l\lVert u_j\rVert_{(a^{(j)})}\dim_K(H^0(X,a^{(j)}\cdot\textbf{L}))^C$$
Let $E=\bigoplus_{i=1}^r L_i$, $\mathbb{P}(E):=\mathrm{\mathcal {P}roj_X}(\mathrm{Sym}(E))$ the projective bundle on $X$ associated with $E$. Let $H=\mathcal O_{\mathbb{P}(E)}(1)$.
Then $$H^0(\mathbb{P}(E),mH)=\bigoplus\limits_{a_1+\cdots+a_r=m} H^0(X,\sum_{j=1}^r a_i L_i)$$
Let $\nm_m$ be the direct sum of $\{\nm_{(a)}\mid a\in \N^r,|a|=m\}$.
\begin{prop}
The normed graded algebra $$\bigoplus_{m\geqslant 0}(H^0(\mathbb{P}(E),mH),\nm_m)$$ is strong $\delta$-supperadditive. 
\end{prop}
\begin{proof}
Let $s_i=\sum\limits_{|a|=m_i} u_a^{(i)}\in H^0(\mathbb{P}(E),m_iH)$ where $i=1,2,\dots,l$, $m_i\in \N$, $u_{a}^{(i)}\in H^0(X,a\cdot \textbf{L})$.
It holds that
$$\lVert s_i\rVert_{m_i}=\max_{|a|=m_i}\{\lVert u_a^{(i)}\rVert_{(a)}\}.$$
Moreover,
$${\begin{aligned}\allowdisplaybreaks\lVert s_1\cdots s_l\rVert_m&=\max\limits_{|\alpha|=m}\left\{\left\lVert \sum\limits_{\substack{a^{(1)}+\cdots+a^{(l)}=a\\ |a^{(i)}|=m_i}}\prod_{i=1}^l u_{a^{(i)}}^{(i)}\right\rVert_{(\alpha)}\right\}\\
&\leqslant\max_{\substack{|a^{(i)}|=m_i\\i=1,\dots,l}}\left\lVert \prod_{i=1}^l u_{a^{(i)}}^{(i)}\right\rVert_{(a^{(1)}+a^{(2)}+\cdots+a^{(l)})}\\
&\leqslant \max_{\substack{|a^{(i)}|=m_i\\i=1,\dots,l}}\prod_{i=1}^l\left\lVert u_{a^{(i)}}^{(i)}\right\rVert_{(a^{(i)})}\dim_K(H^0(X,a^{(i)}\cdot \textbf{L}))^C\\
&\leqslant \prod_{i=1}^l{\max_{|a|=m_i}\{\lVert u_{a}^{(i)}\rVert_{(a)}\dim_K(H^0(X,a\cdot \textbf{L}))^C}\}\\
&\leqslant \prod_{i=1}^l\lVert s_i\rVert_{m_i}\dim_K(H^0(\mathbb{P}(E),m_iH))^C.
\end{aligned}}$$
\end{proof}
\begin{coro}\label{coro_semiample_min_slop}
If $L_1,\cdots,L_r$ are semiample, then there exists constants $S_{\min},T_{\min}\in\R$ such that
$$\mumin(H^0(X,a\cdot{\bf L}),\nm_{(a)})\geq S_{\min}\lvert a\rvert +T_{\min}.$$
\end{coro} 
\begin{proof}
By construction, $H$ is semiample due to the semiampleness of $L_1,\cdots,L_r$. Therefore we conclude the proof by using Proposition \ref{prop_fin_gen_bound}.
\end{proof}
\section{Arithmetic $\chi$-volume and adelic line bundles}\label{sec_adelic_lin}
In this section, we fix the following setting:
\begin{enumerate}
    \item[\textnormal(i)] $K$ is perfect.
    \item[\textnormal(ii)] $K$ is countable or the $\sigma$-algebra $\mathcal A$ is discrete.
    \item[\textnormal(iii)] For every $\omega\in\Omega_\infty$, $\lvert\cdot\rvert_\omega=\lvert\cdot\rvert$ on $\Q$.
    \item[\textnormal(iv)] There exists a measurable subset $\Omega'\in\mathcal A$ such that $\nu(\Omega')\not\in\{0,+\infty\}$.
    \item[\textnormal(v)] $X\rightarrow \mathrm{Spec}K$ is projective and reduced.
\end{enumerate}
\subsection{Adelic line bundles}
Let $H$ be an very ample line bundle whose global sections $E:=H^0(X,H)$ is equipped with a dominated norm family $\xi=\{\nm_{\omega}\}_{\omega\in\Omega}$ i.e. $\overline E=(E,\xi)$ is an adelic vector bundle. For each place $\omega\in\Omega$, we denote by $\xanomega$ the analytification of $X\times_{\mathrm{Spec}K}\mathrm{Spec} K_\omega$ in the sense of Berkovich(see \cite{Berkovich}). As in \cite[2.2.3]{adelic}, the norm $\nm_{\omega}$ induces a \textit{Fubini-Study metric} family $\varphi_{FS, \overline E,\omega}=\{\lvert\cdot\rvert_{\omega}(x)\}$, which is continuous in the sense that for any open subset $U\subset X$ and a section $s\in H^0(X,U)$, the function 
$$x\in U^{an}_\omega\mapsto \lvert s\rvert_\omega(x)$$
is continuous with respect to Berkovich topology, where $U^{an}_\omega$ is the analytification of $U\times_{\mathrm{Spec}K}\mathrm{Spec} K_\omega$.

For arbitrary line bundle $L$, and two continuous metric families $\varphi=\{\varphi_\omega\}$, $\varphi'=\{\varphi'_\omega\}$, we define the distance function $$\operatorname{dist}(\varphi,\varphi'):(\omega\in\Omega)\mapsto\sup_{x\in \xanomega}\lvert1\rvert_{\varphi_{\omega}-\varphi'_{\omega}}(x)$$
where $\lvert1\rvert_{\varphi_{\omega}-\varphi'_{\omega}}(x)$ is a continuous function on $X$ since $\varphi_\omega-\varphi'_\omega$ is a continous metric of $\O_X$.

We say the pair $(H,\varphi)$ of a very ample line bundle and a continuous metric family is an \textit{adelic very ample line bundle} if $\varphi$ is measurable (see \cite[ 6.1.4]{adelic}), and there exists an dominated norm family $\xi$ on the global sections $E=H^0(X,H)$ such that the distance function
$\operatorname{dist}(\varphi,\varphi_{FS,\overline E})$ is $\nu$-dominated.

As the very ample line bundles generates the Picard group $\mathrm{Pic}(X)$, we define the \textit{arithmetic Picard group} $\widehat{\mathrm{Pic}}(X)$ be the abelian group generated by all adelic very ample line bundles. An element in $\widehat{\mathrm{Pic}}(X)$ is called an \textit{adelic line bundle}.

\begin{defi}
    Let $\ovl L=(L,\phi)$ be an adelic line bundle. We say $\ovl L$ is \textit{ample, semiample, nef,} or \textit{big} if so is $L$. We say $\ovl L$ is \textit{semipositive} if for every $\omega\in\Omega$, $\phi_\omega$ is semipositive. We say an adelic line bundle $\ovl L$ is \textit{integrable} if we can write it as $\ovl L=\ovl L_1-\ovl L_2$ where $\ovl L_1$ and $\ovl L_2$ are ample and semipositive adelic line bundles. The group of integrable adelic line bundles is denoted as $\widehat{\mathrm{Int}}(X)$.
\end{defi}
We refer the reader to \cite[4.4.4]{chen2021intersection} for the following:
\begin{prop}
    There exists an intersection pairing on the group $\widehat{\mathrm{Int}}(X)$:
    $$(\ovl L_0,\dots, \ovl L_{d})\in\widehat{\mathrm{Int}}(X)^{d+1}\mapsto \widehat c_1(\ovl L_0)\cdots\widehat c_1(\ovl L_{d})$$
    which is a multilinear and symmetric form.
\end{prop}
\subsection{Arithmetic $\chi$-volume} Let $\ovl L=(L,\phi)$ be an adelic line bundle. The map \[\nm_{\phi_\omega} : H^0(X,L)\ot K_\omega \to \R\] given by
$\displaystyle{s \mapsto \sup_{x\in \xanomega}\lvert s\rvert_{\phi_\omega}(x)}$
is a norm on $H^0(X,L)\ot K_\omega$. Denote by $\xi_{\phi}$ the norm family $\{\nm_{\varphi_\omega}\}$. The \textit{arithmetic pushforward} of $\ovl L$ is defined as $\pi_*(\ovl L)=(H^0(X,L),\xi_{\phi})$, which is an adelic vector bundle due to \cite[Theorem 6.2.18]{adelic}. The $\chi$-volume function $\vol_\chi(\cdot)$ and volume function $\vol(\cdot)$ are defined as
$$\vol_\chi(\overline L):=\limsup_{n\rightarrow +\infty}\frac{\ardeg(H^0(X,nL),\xi_{n\varphi})}{n^{d+1}/(d+1)!},$$
and 
$$\vol(\overline L):=\limsup_{n\rightarrow +\infty}\frac{\ardeg_+(H^0(X,nL),\xi_{n\varphi})}{n^{d+1}/(d+1)!}.$$
We also define the following asymptotic invariants:
\[\begin{cases}
    \displaystyle\mumax^{\mathrm{asy}}(\ovl L):=\limsup_{n\rightarrow+\infty}\frac{\mumax(\pi_*(n\ovl L))}{n},\\
    \displaystyle\mumin^{\inf}(\ovl L):=\liminf_{n\rightarrow+\infty}\frac{\mumin(\pi_*(n\ovl L))}{n},\\
    \displaystyle\mumin^{\sup}(\ovl L):=\limsup_{n\rightarrow+\infty}\frac{\mumin(\pi_*(n\ovl L))}{n}.
\end{cases}\]
\begin{rema}
    For each $n\in\N$, let $\nm_n$ be the ultrametric norm induced by the Harder-Narasimhan filtration of $\pi_*(n\ovl L)$. Then $\{(H^0(X,nL),\nm_n)\}_{n\in\N}$ is a $\delta$-superadditive ultrametrically normed graded linear series. By Proposition \ref{prop_succ_slope}, we can see that our definitions above actually coincides with Definition \ref{def_asymp}. It is worth noting that $\mumax^{\mathrm{asy}}(\ovl L)<+\infty$ due to \cite[Proposition 6.4.4]{adelic}. Moreover, consider adelic line bundles $\ovl L_1,\dots,\ovl L_r$ and $a=(a_1,\dots, a_r)\in\N^{r}$. Let $\nm_{(a)}$ be the ultrametric norm associated to the Harder-Narasimhan filtration of $\pi_*(a_1\ovl L_1+\cdots+a_r\ovl L_r)$. Then $\{H^0(a\cdot {\bf L},\nm_{(a)})\}_{a\in\N^r}$ satisfies the conditions in \ref{subsect_multi_lin_ser}. Therefore Corollary \ref{coro_semiample_min_slop} can be restated as follows:
    Assume $L_1,\cdots, L_r$ are semiample, then there exists constants $S,T$ such that 
    $$\mumin(\pi_*(a_1\ovl L_1+\cdots+a_r\ovl L_r))\geq S\lvert a\rvert+T.$$
\end{rema}
\begin{prop}\label{prop_vol_chi}
    Let $\ovl L$ be an adelic line bundle.
    \begin{enumerate}
        \item[\textnormal{(i)}] Let $h$ be a $\nu$-integrable function on $\Omega$. Then
        \[
            \begin{cases}
            \displaystyle \vol_\chi(\ovl L(h))=\vol_\chi(\ovl L)+(d+1)\mathrm{vol}(L)\int_\Omega h\nu(d\omega),\\
            \displaystyle\mumin^{\inf}(\ovl L(h))=\mumin^{\inf}(\ovl L)+\int_\Omega h\nu(d\omega).
            \end{cases}
        \]
        \item[\textnormal{(ii)}] If $L$ is big and $\mumin^{\inf}(\ovl L)\in\R$, then the sequence $\displaystyle\left\{\frac{\ardeg(\pi_*(n\ovl L))}{n^{d+1}/(d+1)!}\right\}$ converges to $\vol_\chi(\ovl L)$.
        \item[\textnormal{(iii)}] If $\mumin^{\inf}(\ovl L)>-\infty$, then $$\vol_\chi(n\ovl L)=n^{d+1}\vol_\chi(\ovl L).$$
    \end{enumerate}
\end{prop}
\begin{proof}
(i) This is a direct result from definitions and Proposition \ref{prop_deg_posdeg}.
(ii) This is due to Remark \cite[6.3.27]{adelic}.
(iii)
    If $L$ is not big, then $\vol_\chi(\ovl L)=\vol_\chi(n\ovl L)=0$.
    If $L$ is big, we conclude the proof by using (ii).
\end{proof}
\subsection{An arithmetic Hilbert-Samuel formula} In this subsection, we extend the result in \cite[Theorem 5.5.1]{huayimoriwaki_equi} to a semiample and semipositive adelic line bundle.
\begin{lemm}\label{lemm_pos_notbig}
    Let $\ovl L$ be an adelic line bundle such that $L$ is not big.
    Then $\vol(\ovl L)=0$. Moreover we have the following
    \begin{enumerate}
        \item[\textnormal{(a)}] If $\mumin^{\inf}(\ovl L)\in\R$, then $\vol_\chi(\ovl L)=\vol(\ovl L)=0$.
        \item[\textnormal{(b)}] If $\ovl L$ is nef and semipositive, then $\widehat c_1(\ovl L)^{d+1}\leq 0$.
    \end{enumerate}
\end{lemm}
\begin{proof}
    The volume $\vol(\ovl L)=0$ is trivial due to the fact that $$\vol(\ovl L)\leq \limsup \frac{\mumax^{\mathrm{asy}}(\pi_*(n\ovl L))h^0(X,nL)}{n^{d+1}/(d+1)!}\leq (d+1)\mumax^{\mathrm{asy}}(\ovl L)\mathrm{vol}(L).$$
    
(a) Since the asymoptotic minimal slope is bounded, we can choose a $\nu$-integrable function $h$ such that $\vol_\chi(\ovl L(h))=\vol(\ovl L(h))=0.$ Then we obtain the statement by using Proposition \ref{prop_vol_chi} (i).

(b) This is a direct result of \cite[Theorem 6.5.1]{huayimoriwaki_equi}.
\end{proof}
\begin{lemm}\label{lemm_vol_ineq}
    Let $\ovl L=(L,\phi)$ and $\ovl M=(M,\psi)$ be adelic line bundles satisfying the following conditions:
    \begin{itemize}
        \item[\textnormal{(1)}] $\mumin^{\inf}(\ovl M)\in\R$.
        \item[\textnormal{(2)}] There exists a non-zero section $s\in H^0(X,L-M)$ such that $\ardeg_{\xi_{\phi-\psi}}(s)\geq 0$.
    \end{itemize}
    Then  $$\vol_\chi(\ovl L)\leq \vol_\chi(\ovl M)-(d+1)\mumin^{\inf}(\ovl M)(\mathrm{vol}(L)-\mathrm{vol}(M)).$$
\end{lemm}
\begin{proof}
    For any $m\in \N_+$, consider the injective morphism
    $$H^0(X,mL)\xrightarrow{\cdot s^{m}}H^0(X,mm).$$
    We denote by $I(m)$ the image of above morphism, and by $Q(m)$ the cokernel. We denote by $\ovl I(m)$ the sub-bundle of $\pi_*(m\ovl M)$, and by $\ovl Q(m)$ the quotient bundle. Since for each $\omega\in\Omega$, the morphism $H^0(X,mL)\ot {K_\omega}\rightarrow I(m)\ot {K_\omega}$ is of operator norm$\leq \lVert s^m\rVert_{m(\phi-\psi)_\omega}=\lVert s\rVert_{(\phi-\psi)_\omega}^m,$ we see that $$\ardeg(\pi_*(m\ovl L))\leq \ardeg(\ovl I(m))-m\ardeg_{\xi_{\phi-\psi}}(s)h^0(X,mL)\leq \ardeg(\ovl I(m)).$$
    Since \begin{align*}
        \ardeg(\ovl I(m))&\leq \ardeg(\pi_*(m\ovl M))-\ardeg(\ovl Q(m))\\
        &\leq \ardeg(\pi_*(m\ovl M))-\mumin(\pi_*(m\ovl M))\dim_K Q(m),
    \end{align*}
    by definition we have
    $$\vol_\chi(\ovl L)\leq \vol_\chi(\ovl M)-(d+1)\mumin^{\inf}(\ovl M)(\mathrm{vol}(L)-\mathrm{vol}(M)).$$
\end{proof}
\begin{theo}\label{theo_semi_hil_sam}
    Let $\ovl L$ be a semiample and semipositive adelic line bundle. 
    It holds that $$\vol_\chi(\ovl L)=\widehat c_1(\ovl L)^{d+1}$$
\end{theo}
\begin{proof}
    {\bf Step 1.} We first show that $\vol_\chi(\ovl L)\leq \widehat c_1(\ovl L)^d$.
    Let $\ovl A=(A,\psi)$ be an ample and semipositive adelic line bundle. We may assume that there exists a non-zero section $s\in H^0(X,A)$ such that $\ardeg_{\xi_{\psi}}(s)\geq 0$. Indeed, we can assume that $H^0(X,A)\not=0$, and take arbitrary non-zero $s\in H^0(X,A)$. If $\ardeg_{\xi_{\psi}}(s)<0$, we take a non-negative $\nu$-integrable function $h$ such that $$\ardeg_{\xi_{\psi+h}}(s)=\ardeg_{\xi_{\psi}}(s)+\int_{\Omega}h\nu(\omega)\geq 0,$$ then we can replace $\psi$ by $\psi+h$.
    
    For any $n\in\N_+$, by Lemma \ref{lemm_vol_ineq},
    \begin{align*}
        \vol_\chi(n\ovl L)\leq \vol_\chi(n\ovl L+\ovl A)-(d+1)\mumin^{\inf}(n\ovl L+\ovl A)(\mathrm{vol}(nL+A)-\mathrm{vol}(nL))\\
        = \widehat c_1(n\ovl L+\ovl A)^{d+1}-(d+1)\mumin^{\inf}(n\ovl L+\ovl A)(\mathrm{vol}(nL+A)-\mathrm{vol}(nL))
    \end{align*}
    where the equality is due to the Hilbert-Samuel for ample and semipositive adelic line bundles \cite[Theorem 5.5.1]{adelic}.
Then by the homogeneity of $\vol_\chi(\cdot)$, we obtain that
\begin{align*}
    \vol_\chi(&\ovl L)=\frac{1}{n^{d+1}}\vol_\chi(n\ovl L)\\&\leq \frac{1}{n^{d+1}}\widehat c_1(n\ovl L+\ovl A)^{d+1}-(d+1)\frac{\mumin^{\inf}(n\ovl L+\ovl A)(\mathrm{vol}(nL+A)-\mathrm{vol}(nL))}{n^{d+1}}.
\end{align*} Notice that $\mumin^{\inf}(n\ovl L+\ovl A)\geq O(n)$ and $\displaystyle\lim\limits_{n\rightarrow +\infty}\frac{\mathrm{vol}(nL+A)-\mathrm{vol}(nL)}{n^{d}}=0$. We obtain the inequality by letting $n\rightarrow+\infty$.

{\bf Step 2.} We show the reverse inequality $\vol_\chi(\ovl L)\geq \widehat c_1(\ovl L)^{d+1}$. We first assume that $L$ is not big, in which case we have $0=\vol_\chi(\ovl L)\leq \widehat c_1(\ovl L)^{d+1}\leq 0$ due to Lemma \ref{lemm_pos_notbig} and Step 1. Now we assume that $L$ is big. Then for $r\gg0$, there exists a non-zero section $t\in H^0(X,rL-A)$. We can take a $\nu$-integrable function $g$ such that $\ardeg_{\xi_{r\phi-\psi+g}}(t)\geq 0.$ Let $\displaystyle C_g:=\int_\Omega h\nu(d\omega).$

By the similar argument as in Step 1, we can obtain that for any fixed $n\in\N_+$, \begin{align*}\allowdisplaybreaks
    &\widehat c_1(n\ovl L+\ovl A)^{d+1}=\vol_\chi(n\ovl L+\ovl A)\\&\leq \vol_\chi((n+r)\ovl L+(\O_X,g))\\&\kern5em -(d+1)(\mumin^{\inf}((n+r)\ovl L)+C_g)(\mathrm{vol}((n+r)L)-\mathrm{vol}(nL+A))\\&=(n+r)^{d+1}\vol_\chi(\ovl L)+(d+1)(n+r)^d\mathrm{vol}(L)C_g\\
    &\kern5em -(d+1)(\mumin^{\inf}((n+r)\ovl L)+C_g)(\mathrm{vol}((n+r)L)-\mathrm{vol}(nL+A)).
\end{align*}
By taking quotient over $(n+r)^{d+1}$ of above and letting $n\rightarrow +\infty$, we obtain the reverse inequality.
\end{proof}
\begin{rema}
    We can see that the same proof works if the following conditions are satisfied:
    \begin{enumerate}
        \item $\ovl L$ is nef and semipositive, and $\mumin^{\inf}(\ovl L)\in\R$,
        \item $\mumin^{\inf}(n\ovl L+\ovl A)\geq O(n)$ for some ample and semipositive adelic line bundle $\ovl A$.
    \end{enumerate}
\end{rema}
\subsection{Birational invariance}

\begin{defi}
Let $f:X'\rightarrow X$ be a morphism of geometrically reduced projective $K$-varieties. Let $\ovl L:=(L,\phi)$ be an adelic line bundle over $X$. We define the \textit{pull-back} $f^*\ovl L$ of $\ovl L$ as $(L,\{f^*\phi_\omega\}_{\omega\in\Omega}).$ We refer the reader to \cite[Definition 2.2.4]{huayimoriwaki_equi} for the definition of the pull-back of a metric.
\end{defi}

In the following, we give a proof that $\vol_\chi(\cdot)$ is a birational invariant without using Stein-factorization.
\begin{theo}\label{theo_bir_inv}
Let $f:X'\rightarrow X$ be a birational morphism of geometrically reduced projective $K$-varieties. Let $\ovl L$ be an adelic line bundles such that $\mumin^{\inf}(f^*\ovl L)>-\infty$. Then 
$$\vol_\chi(f^*\ovl L)=\vol_\chi(\ovl L).$$
\end{theo}
\begin{proof}
    For any $\omega\in \Omega$, we denote by $f_\omega$ the base change $f\times (\mathrm{Spec}K_\omega\rightarrow \mathrm{Spec}K)$. Let $s\in H^0(X,mL)\ot K_\omega$, for any $y\in (X')_\omega^{\mathrm{an}}$ and $x=f_\omega^{\mathrm{an}}(y)$, by definition we have $\lvert f^*(s)\rvert_{f^*(m\phi_\omega)}(y)=\lvert s\rvert_{m\phi_\omega}(x)\ot_{\epsilon,\pi}\widehat \kappa(y)=\lvert s\rvert_{m\phi_\omega}(x)$, where the second equality is due to \cite[Proposition 1.3.1]{adelic}.
    
    Therefore we can see that $\xi_{m\phi}$ coincides with the norm family induced by $\xi_{f^*(m\phi)}$ and $H^0(X,mL)\rightarrow H^0(X',f^*(mL))$. Let $$Q(m,L):=H^0(X',f^*(mL))/H^0(X,mL).$$ We denote by $Q(m\ovl L)$ the adelic vector bundle $(Q(mL),\xi_{f^*(m\phi),\quot})$ which is the quotient of $\pi'_*(f^*(m\ovl L))$ over $\pi_*(m\ovl L)$. Therefore by Proposition \ref{prop_exact}, we have 
    \begin{align*}
        0\leq \ardeg(&\pi'_*(f^*(m\ovl L)))-(\ardeg(\pi_*(m\ovl L))+\ardeg(Q(m\ovl L)))\\
        &\leq \frac{\nu(\Omega_\infty)}{2} h^0(X',f^*(mL))\ln(h^0(X',f^*(mL)).
    \end{align*}
    So it suffices to prove that $\ardeg(Q(m\ovl L))=o(m^{d+1}).$
    
    {\bf Step 1.} We first assume that $\ovl L$ is ample and semipositive, then $f^*(m\ovl L)$ is semiample and semipositive.
    By using Theorem \ref{theo_semi_hil_sam}, we obtain that $\ardeg(Q(m\ovl L))=o(m^{d+1}).$

    {\bf Step 2.} Now we consider the general case. Let $\ovl A=(A,\psi)$ be an adelic ample and semipositive line bundle such that there exists a non-zero section $s\in H^0(A-L)$ not vanishing at each associated point of $f_*(\O_{X'}/\O_X)$.
    Consider the commutative diagram:
    $$\begin{tikzcd}
    H^0(X',f^*(mL)) \arrow[r, "f^*(s^m)"] \arrow[d]
    & H^0(X',f^*(mA)) \arrow[d,  ] \\
    Q(m L) \arrow[r, "r_m"]
    & Q(m A).
    \end{tikzcd}$$
    For each $\omega\in\Omega$, we can see that the injective morphism $(r_m)_{K_\omega}$ is of operator norm $\leq \lVert f^*(s^m)\rVert_{f^*(m(\psi-\phi))_\omega}=\lVert s^m\rVert_{m(\psi-\phi)_\omega}=\lVert s\rVert_{(\psi-\phi)_\omega}^m.$ Therefore
    \begin{align*}
        \mumin(Q(m\ovl L))\dim_K Q(mL)&\leq \\ \ardeg(Q(m\ovl L))\leq -m&\ardeg_{\xi_{\psi-\phi}}(s)(\dim_K(Q(mL)))+\ardeg(Q(m\ovl A))\\&-\mumin(Q(m\ovl A))\left(\dim_K(Q(mA))-\dim_K(Q(mL))\right)
    \end{align*}
    Notice that \[\begin{cases}\displaystyle
        \liminf\limits_{m\rightarrow+\infty}\frac{\mumin(Q(m\ovl A))}{m}\geq \liminf\limits_{m\rightarrow+\infty}\frac{\mumin(\pi'_*(f^*(m\ovl A)))}{m}=\mumin^{\mathrm{asy}}(f^*(\ovl A))> -\infty,\\
        \displaystyle\liminf\limits_{m\rightarrow+\infty}\frac{\mumin(Q(m\ovl L))}{m}\geq \liminf\limits_{m\rightarrow+\infty}\frac{\mumin(\pi'_*(f^*(m\ovl L)))}{m}=\mumin^{\inf}(f^*(\ovl L))> -\infty,\\
        \displaystyle\dim_K Q(mA),\dim_K Q(mL)\sim o(m^d).
    \end{cases}\]
    We obtain that $Q(m\ovl L)\sim o(m^{d+1})$ which implies that $\vol_\chi(f^*\ovl L)=\vol_\chi(\ovl L).$
\end{proof}
\subsection{Concavity of $\chi$-volume}
Let $\ovl L$ be a big adelic line bundle. Let $\Delta(L)$ be the Okoukov body associated to the linear series $\{H^0(X,nL)\}_{n\in\N}$ defined as in \ref{subsect_Okounkov}. Let $G_{\ovl L}:\Delta(L)\rightarrow \R$ be the concave transform given by the ultrametric normed linear series associated to the Harder-Narasimhan filtrations of $\pi_*(n\ovl L)$. Then we can see that if $\mumin^{\inf}(\ovl L)\geq -\infty$, then $$\vol_\chi(\ovl L)=\int_{\Delta(L)}G_{\ovl L}(x)dx$$
due to Proposition \ref{prop_Okounkov}. Let $\ovl L_1$ and $\ovl L_2$ be big adelic line bundles. Then $\Delta(L_1)+\Delta(L_2)\subset \Delta(L_1+L_2)$ and \begin{equation}\label{ineq_concave_trans}
    G_{\ovl L_1+\ovl L_2}(x+y)\geq G_{\ovl L_1}(x)+G_{\ovl L_2}(y)
\end{equation} for any $x\in \Delta(L_1)$ and $y\in\Delta(L_2)$ due to Proposition \ref{prop_Okounkov}(5). Therefore we have the following concavity theorem:
\begin{prop}\label{prop_concav}
    let $\ovl L_1$ and $\ovl L_2$ be big adelic line bundles such that 
    \begin{enumerate}
        \item[\textnormal{(1)}] $\Delta(L_1)+\Delta(L_2)=\Delta(L_1+L_2)$,
        \item[\textnormal{(2)}] $\mumin^{\inf}(\ovl L_1),\mumin^{\inf}(\ovl L_2),\mumin^{\inf}(\ovl L_1+\ovl L_2)\in\R$
    \end{enumerate}
    Then 
    $$\frac{\vol_\chi(\ovl L_1+\ovl L_2)}{\mathrm{vol}(L_1+L_2)}\geq\frac{\vol_\chi(\ovl L_1)}{\mathrm{vol}(L_1)}+\frac{\vol_\chi(\ovl L_2)}{\mathrm{vol}(L_2)}$$
\end{prop}
\begin{proof}
    This is a direct result of \eqref{ineq_concave_trans}.
\end{proof}
\section{Continuous extension of $\chi$-volume}\label{sect_adl_R_div}
We keep the same assumption as in section \ref{sec_adelic_lin}.
We further assume that $X$ is geometrically integral and normal.
For a Cartier divisor $D$ on $X$, we define the \textit{$D_\omega$-Green function} $g_\omega$ to be an element of the set
$$C_{gen}^0(\xanomega):=\{f\text{ is a continuous function on }U\mid \emptyset\not= U\mathop{\subset}_{open} \xanomega\}/\sim$$
where $f\sim g\text{ if they are identical on some } V_{\omega}^{an}$ for some non-empty open subset $V\subset X$ such that for any local equation $f_D$ of $D$ on $U$, $\ln|f_D|+g_\omega$ is continuous on $U^{an}_\omega$. It's easy to see that each $D_\omega$-Green function induce a continuous metric on the corresponding line bundle (see \cite[Section 2.5]{adelic} for details). Moreover, we say a pair $(D,g=\{g_\omega\})$ of Cartier divisor $D$ and Green function family $g$ is an \textit{adelic Cartier divisor} if the it corresponds to an adelic line bundle. We denote by $\widehat{\Div}(X)$ the group of all adelic Cartier divisor. Let $\mathbb{K}=\Q\textit{ or }\R$, then we can define the set of adelic $\mathbb{K}$-Cartier divisors as
$$\widehat{\Div}_{\mathbb{K}}(X)=\widehat{\Div}(X)\otimes_\Z \mathbb{K}/\sim$$
where "$\sim$" is the equivalence relationship generated by
$\sum (0,g_i)\otimes k_i\sim (0,\sum g_i k_i)$, where $g_i$'s are continuous funciton families and $k_i\in \mathbb{K}$.

For any $\R$-Caritier divisor $D$, we can define the global section space as follows:
$$H_\R^0(D):=\{f\in K(X)^\times\mid \mathrm{div}(f)+D\geqslant_\R 0\}\cup\{0\}.$$
The conditions of Green function family being dominated and measurable will lead us to the following result:
\begin{theo}\label{theo_push_adelic_R_div}
For any $(D,g)\in \widehat{\Div}_\R(X)$ and $\omega\in\Omega$, we consider a norm $\lVert\cdot\rVert_{g_\omega}$ defined as
$$\lVert f\rVert_{g_\omega}:=\sup_{x\in X^{\mathrm{an}}_\omega}\{\exp(-g_\omega(x)f(x)\}$$
for $f\in H^0_\R(D)\otimes_{K}K_\omega$. Let $\xi_g$ denote the norm family $\{\lVert\cdot\rVert_{g_\omega}\}_{\omega\in\Omega}$. Then the pair $(H^0_\R(D),\xi_g)$ is an adelic vector bundle.
\end{theo}
\begin{proof}
See \cite[Theorem 6.2.18]{adelic}.
\end{proof}
Now we can extend the definitions of the asymptotic invariants including $\vol_\chi(\cdot),\vol(\cdot),\mumax^{\mathrm{asy}}(\cdot),\mumin^{\inf}(\cdot)$ and $\mumin^{\sup}(\cdot)$ to adelic $\R$-Cartier divisors. In particular, $\vol(\cdot)$ is continuous in the sense of \cite[Theorem 6.4.24]{adelic}.

\begin{prop}
    Let $\ovl D$ be an adelic $\R$-Cartier divisor such that $\mumin^{\sup}(\ovl D)\in\R$, then $\vol_\chi(\alpha\ovl D)=\alpha^{d+1}\vol_\chi(\ovl D)$ for $\alpha\in\Q_+$.
\end{prop}
\begin{proof}
    {\bf Step 1.} We first prove that the equality holds for $\alpha\in\N_+$. Notice that $\mumin^{\sup}(\alpha\ovl D)\leq \alpha\mumin^{\sup}(\ovl D)$ by definition. Let $h$ be a $\nu$-integrable function such that $\displaystyle\int_\Omega h\nu(d\omega)>-\mumin^{\sup}(\ovl D)$. Then \begin{align*}
    \vol_\chi(\alpha (\ovl D+(0,h)))&=\vol
    (\alpha(\ovl D+(0,h)))\\&=\alpha^{d+1}\vol(\ovl D+(0,h))=\alpha^{d+1}\vol_\chi(\ovl D+(0,h))
    \end{align*}
    where the first and third equality are due to Proposition \ref{prop_Okounkov}(4), and the second equality is due to \cite[Corollary 6.4.14]{adelic}. Since $\displaystyle\vol_\chi(\alpha(\ovl D+(0,h)))=\vol_\chi(\alpha\ovl D)+(d+1)\mathrm{vol}(\alpha D)\int_\Omega \alpha h\nu(d\omega)=\vol_\chi(\alpha\ovl D)+(d+1)\alpha^{d+1}\mathrm{vol}(D)\int_\Omega h\nu(d\omega)$, we obtain that $\vol_\chi(\alpha \ovl D)=\alpha^{d+1}\vol_\chi(\ovl D).$

    {\bf Step 2.} Now assume that $\alpha=\beta/\gamma$ for $\beta,\gamma\in \N_+$. Then $\beta^{d+1}\vol_\chi(\ovl D)=\vol_\chi(\beta \ovl D)=\vol_\chi(\alpha \gamma\ovl D)=\gamma^{d+1}\vol_\chi(\alpha\ovl D),$ we are done.
\end{proof}
\begin{theo}
    If in addition $\dim X=1$ or $X$ is toric, then for adelic $\Q$-Cartier divisors $\ovl D_1,\cdots, \ovl D_r$ with $D_1,\dots, D_r$ being ample, the function $$a=(a_1,\dots, a_r)\in\Q^{r}\mapsto \frac{\vol_\chi(a_1\ovl D_1+\cdots+a_r\ovl D_r)}{\mathrm{vol}(a_1 D_1+\cdots+a_r D_r)}$$
    is concave.
\end{theo}
\begin{proof}
    Let $a=(a_1,\dots, a_r),b=(b_1,\dots, b_r)\in\Q^{r}$ and $\lambda\in (0,1)\cap \Q$.
    We want to show that 
    \begin{align*}
        \frac{\vol_\chi(\sum (\lambda a_i+(1-\lambda)b_i)\ovl D)}{\mathrm{vol}(\sum(\lambda a_i+(1-\lambda)b_i)D)}&\geq \lambda\frac{\vol_\chi(\sum a_i\ovl D)}{\mathrm{vol}(\sum a_i D)}+(1-\lambda)\frac{\vol_\chi(\sum b_i\ovl D)}{\mathrm{vol}(\sum b_i D)}.\\
        &=\frac{\vol_\chi(\sum \lambda a_i\ovl D)}{\mathrm{vol}(\sum \lambda a_i D)}+\frac{\vol_\chi(\sum (1-\lambda )b_i\ovl D)}{\mathrm{vol}(\sum (1-\lambda)b_i D)}.
    \end{align*}
    By homogeneity, we may assume that $D_1,\dots, D_r$ are very ample, and $\lambda a_i,(1-\lambda) b_i$ are integers. Then we are done by Proposition \ref{prop_concav} and the additivity of Okounkov bodies for curves and toric varieties\cite[Theorem 3.1]{kaveh2011note}.
\end{proof}

\subsection{Continuity of $\chi$-volume}
\begin{lemm}\label{lem_conti_ext}
Let $\ovl D_1=(D_1,g_1),\dots,\ovl D_r=(D_r,g_r)$ be adelic Cartier divisors on $X$. We denote by \[
\begin{cases}
C_\R(\ovl D_1,\dots,\ovl D_r):=\ovl D_1 \cdot\R_{\geq 0}+\cdots+\ovl D_r\cdot \R_{\geq 0},\\
C_\Q(\ovl D_1,\dots,\ovl D_r):=\ovl D_1 \cdot\Q_{\geq 0}+\cdots+\ovl D_r\cdot \Q_{\geq 0},
\end{cases}
\]
the polyhedral cones generated by $\ovl D_1,\dots, \ovl D_r$ over $\R_{\geq 0}$ and $\Q_{\geq 0}$ respectively.
Assume that there exists a constant function $\lambda(\cdot)$ on the cone $C_\R(\ovl D_1,\dots,\ovl D_r)$
such that $$\mumin^{\sup}(\ovl D)> \lambda(\ovl D)$$ for any $\ovl D\in C_\Q(\ovl D_1,\dots,\ovl D_r)$. Then there exists a continuous function $\vol_{I}(\cdot)$ on $C_\R(\ovl D_1,\dots,\ovl D_r)$ such that
$$\vol_I(\ovl D)=\vol_{\chi}(\ovl D)$$ for any $\ovl D\in C_\Q(\ovl D_1,\dots,\ovl D_r)$. 
\end{lemm}
\begin{proof}
We take a $\nu$-integrable function $f$ such that $\displaystyle\int_{\Omega} h\nu(\omega)=1.$ We set that $$\vol_I(\ovl D):=\vol(\ovl D-\lambda(\ovl D)(0,h))+(d+1)\mathrm{vol}(D)\lambda(\ovl D).$$
This is obviously a continuous function due to the continuity of $\vol(\cdot)$, $\mathrm{vol}(\cdot)$ and $\lambda(\cdot)$.
Moreover, for any $\ovl D\in C_\Q(\ovl D_1,\dots,\ovl D_r)$, since $$\mumin^{\sup}(\ovl D-\lambda(\ovl D)(0,h))=\mumin^{\sup}(\ovl D)-\lambda(\ovl D)>0,$$
we have $$\vol(\ovl D-\lambda(\ovl D)(0,h))=\vol_\chi(\ovl D-\lambda(\ovl D)(0,h))=\vol_\chi(\ovl D)-(d+1)\mathrm{vol}(D)\lambda(\ovl D)$$, we thus conclude the proof.
\end{proof}

\begin{theo}\label{theo_chi_semiample_con}
Let $\overline D_1=(D_1,g_1),\dots, \overline D_r=(D_r,g_r)$ be adelic $\Q$-Cartier divisors on $X$ such that $D_i$ are semiample. Then $\vol_\chi(\cdot)$ is continuous in $C_\Q(\ovl D_1, \dots, \ovl D_r)$. Moreover $\vol_\chi(\cdot)$ can be continuously extended to a function on $C_\R(\ovl D_1, \dots, \ovl D_r)$.
\end{theo}
\begin{proof}
Due to the homogeneity, we may assume that all $D_i$ are integral Cartier divisors. 
Hence there exists constants $S$ and $T$ depending on $\overline D
_1,\dots,\overline
D_r$, such that
$$\mumin(H^0(X,\sum_{i=1}^r n_i D_i),\xi_{\sum_{i=1}^r n_i g_i})\geqslant T+S\sum_{i=1}^{r}n_i{,}$$
where $n_i\in \N$.
We write that $a_i=p_i/q_i$ where $p_i$ and $q_i$ are coprime integers and $q_i>0$. Let $q=\prod_{i=1}^r q_i$. Then it holds that
$$\mumin(H^0(X,mq(\sum_{i=1}^r a_i D_i)),\xi_{mq(\sum_{i=1}^r a_i g_i)})\geqslant T+Smq(\sum_{i=1}^r a_i)$$
for every $m\in\N$.
Therefore
$$\mumin^{\sup}(\sum_{i=1}^r a_i \overline D_i)\geqslant S(\sum_{i=1}^r a_i).$$
We then conclude the proof by using Lemma \ref{lem_conti_ext}.
\end{proof}

\bibliography{mybibliography}
\bibliographystyle{smfplain}
\end{document}